\documentclass{article}

\usepackage[utf8]{inputenc} % allow utf-8 input
\usepackage[T1]{fontenc}    % use 8-bit T1 fonts
\usepackage{hyperref}       % hyperlinks
\usepackage{url}            % simple URL typesetting
\usepackage{booktabs}       % professional-quality tables
\usepackage{amsfonts}       % blackboard math symbols
\usepackage{nicefrac}       % compact symbols for 1/2, etc.
\usepackage{microtype}      % microtypography
\usepackage{lipsum}
\usepackage{graphicx}
\usepackage{amsmath}
\usepackage{amssymb}
\usepackage{amsthm}
\usepackage{color}
\graphicspath{ {./images/} }
\newtheorem{theorem}{Theorem}

\newtheorem{corollary}{Corollary}

\newtheorem{definition}{Definition}
\newtheorem{example}{Example}

\newtheorem{lemma}{Lemma}

\newtheorem{proposition}{Proposition}

\usepackage{verbatim}

\title{Algebraic Aspects of Combined Matrices}

\author{ Primitivo B. Acosta-Hum\'anez\footnote{ \texttt{pacosta-humanez@uasd.edu.do}} ,  Randy Leonardo\footnote{ \texttt{rleonardo36@uasd.edu.do}}  and  M\'aximo Santana\footnote{ \texttt{msantana22@uasd.edu.do - Corresponding author} }\\ 
   \small{Instituto de Matem\'atica, Universidad Aut\'onoma de Santo Domingo, Dominican Republic}}
   \date{}
\begin{document}
\maketitle

\begin{abstract}
In this work, we present algebraic results concerning the combined matrices $\mathcal{C}(A)$, where the entries of $A$ belong to a number field $K$ and $A$ is a non-singular matrix. In other words, $A$ is a $n\times n$ matrix belonging to the General Linear Group over $K$, denoted by $\mathrm{GL}_n(K)$. We also analyze the case in which matrix $A$ belongs to algebraic subgroups of $\mathrm{GL}_n(K)$, such as the unimodular group, where $A^2$ is a $n\times n$ matrix belonging to the Special Linear Group, denoted by $\mathrm{SL}_n(K)$, triangular groups, diagonal groups, among others. In particular, we thouroughly examine the cases $n=2$ and $n=3$ for symmetric and non-symmetric matrices, providing explicit diagonalization of $\mathcal{C}(A)$, which includes characteristic polynomials with their eigenvalues and eigenfactors.\medskip

\end{abstract}

\section{Introduction}

The study of combined matrices has garnered significant interest among researchers in recent years. Recently, the third author, in collaboration with researchers from Spain, published new results on combined matrices (see\cite{bggss23,bggs16,gggss21} and references therein). Previous works related to combined matrices are provided in references \cite{aps19, bggs14,bgs22,f17,fm11}.

The combined matrix of a non-singular matrix \(A\) is defined as \(C(A) = A \circ A^{-T}\), where \(\circ\) denotes the Hadamard product of matrices. Specifically, we are interested in cases where the elements of the matrix \(A\) are rational numbers. Thus, \(A\) is an \(n \times n\) matrix belonging to the General Linear Group over \(\mathbb{Q}\), denoted by \(GL_n(\mathbb{Q})\). We also analyze the case in which the determinant of \(A\) is \(\pm1\) and the entries of \(A\) are integer numbers, i.e., \(A^2\) is an \(n \times n\) matrix belonging to the Special Linear Group \(SL_n(\mathbb{Z})\).

We provide results concerning eigenvalues, eigenvectors, and Galois groups related to combined matrices, i.e., we explore the algebraic properties of combined matrices \(C(A)\) that include the set of eigenvalues and eigenvectors of \(C(A)\) and the Galois group of its characteristic polynomial. We recall that the trace of a square matrix \(A\) is denoted by \(\text{Tr}(A)\) and its determinant by \(\det A\). A matrix \(A\) is singular if \(\det A = 0\); otherwise, it is called non-singular. The canonical basis of \(\mathbb{R}^n\) is denoted by \(\mathbf{B} = \{ e_1, e_2, \ldots, e_n \}\), where \(e_1 = (1, 0, \ldots, 0)\), \(e_2 = (0, 1, \ldots, 0)\), \ldots, \(e_n = (0, 0, \ldots, 1)\).

Now we set the theoretical background necessary to understand the rest of the paper. 
\begin{definition}Let \(A\) be an \(n \times n\) square matrix. The characteristic polynomial of \(A\), denoted by \(p_A(\lambda)\), is the polynomial of degree at most \(n\) obtained as \(p_A(\lambda) = \det(A - \lambda I_n)\), where \(I_n\) is the \(n \times n\) identity matrix. The roots of the characteristic polynomials are the set of eigenvalues of \(A\), denoted by \(\Lambda(A) := \{ \lambda \in \mathbb{C} : p_A(\lambda) = 0 \}\). The set of eigenvectors of \(A\) is denoted by \(\mathcal{V}(A)\). Minors of \(A\) are given by \(A_{ij}\), and cofactors of \(A\) are given by \((-1)^{i+j}A_{ij}\), where \(A_{ij}\) is the determinant of the submatrix of \(A\) obtained by deleting the \(i\)-th row and \(j\)-th column.
\end{definition}
According to \cite{caa13} we give the following definitions.
\begin{definition}
Consider a set $G$ and an operation $(\cdot)$. The pair $(G,\cdot)$ is a group if the following is satisfied
\begin{enumerate}
    \item Closure property: $\forall a,b\in G, \,\, a\cdot b\in G$
    \item Associative property: $\forall a,b,c\in G, \,\, a\cdot (b\cdot c)=(a\cdot b)\cdot c$
    \item Identity element: $\exists e\in G$ such that $\forall a\in G, \,\,a\cdot e=e\cdot a=a$
    \item Inverse element: $\forall a\in G,\,\, \exists b\in G$ such that $a\cdot b=b\cdot a=e$
\end{enumerate}
If \((G, \cdot)\) satisfies the closure property, we say it is a Magma. If \((G, \cdot)\) is a Magma and satisfies the associative property, we say it is a Semigroup. If \((G, \cdot)\) is a Semigroup and satisfies the identity element property, we say it is a Monoid. If \((G, \cdot)\) is a Monoid and satisfies the inverse element property, we say it is a Group. If \((G, \cdot)\) is a Group and satisfies the commutative property (\(\forall a, b \in G, \ a \cdot b = b \cdot a\)), we say it is an Abelian Group. By \(S_n\), we denote the symmetric group of order \(n!\).

\end{definition}
\begin{definition}
An algebraic number field, also called a number field, is a finite degree field extension of the field of rational numbers \(\mathbb{Q}\), where the degree field extension refers to the dimension of the field as a vector space over \(\mathbb{Q}\).
\end{definition}
\begin{definition}
Consider the number fields $(K,+,\cdot)$ and  $(L,+,\cdot)$. The application $f:\,\,K \longrightarrow L$ is:
\begin{enumerate}
    \item Homomorphism whether $f(a+b)=f(a)+f(b)$ and $f(a\cdot b)=f(a)\cdot f(b)$ for all $a,b\in K$.
    \item Monomorphism whether $f$ is Homomorphism and it is injective.
    \item Epimorphism whether $f$ is Homomorphism and it is surjective.
    \item Isomorphism whether $f$ is Monomorphism and Epimorphism.
    \item Endomorphism whether $f$ is Homomorphism and $K=L$.
    \item Automorphism whether $f$ is Isomorphism and Endomorphism.
\end{enumerate}
\end{definition}
\begin{definition}
Consider a polynomial $p\in K[x]$, where $K$ is a field. Let $L/K$ be the extension of $L$ over $K$. The automorphism $\sigma:\,\, L/K\longrightarrow L/K$ is called $K$-automorphism of $L$ whether $\sigma|_K=e$, i.e., $\sigma$ restricted to the elements of $K$ is the identity group. The Galois group $\mathrm{Gal}(L/K)$ is the group of $K$-automorphisms of $L$.
\end{definition}
According to \cite{ac16} the following are the definitions of Reversing of matrices adapted to $n\times n$ matrices.
\begin{definition}
Let $A$ be a $n\times n$ matrix, $\delta_{i,j}=0$ for all $i\neq j$ and $\delta_{i,i}=1$. Reversing of $A$ is given by:
\begin{eqnarray*}\label{drm}
    \mathcal{R}(A)=M_{\mathcal{R}_c}AM_{\mathcal{R}_r},\\ M_{\mathcal{R}_r}=
    [\delta_{i,n-j+1}]_{n\times n},\\  M_{\mathcal{R}_c}=[\delta_{n-i+1,j}]_{n\times n},\\ M_{\mathcal{R}_c}\cdot M_{\mathcal{R}_r}=M_{\mathcal{R}_r}\cdot M_{\mathcal{R}_c}=I_n.
    \end{eqnarray*}
\end{definition}

The following is the definition of combined matrices.
\begin{definition}\label{defmc}
Let $A$ be a non-singular matrix. The combined matrix of $A$ is 
\begin{equation}\label{eq:dcm}
\mathcal{C}(A)=A\circ A^{-T}=A\circ (A^{-1})^T,
\end{equation}
where in general $B\circ C$ is the Hadamard (entrywise) product of $B$ and $C$.
\end{definition}
\section{Results of combined matrices}
In this section we consider results concerning algebraic aspects of combined matrices of $n\times n$ matrices. Propositions \ref{prop1}, \ref{prop2}, \ref{prop3} and \ref{prop4} are immediate consequences of Definition \ref{defmc} for $n\in\{2,3\}$ and direct computations regarding properties studied in a basic course of linear algebra for undergraduate level.
\subsection{Combined Matrices in $GL(2,K)$}

Here we analyze some algebraic aspects of characteristic polynomials of non-singular matrices of size $2\times 2$. Consider $$A=\begin{pmatrix}a_{11}&a_{12}\\a_{21}&a_{22}\end{pmatrix},\quad \mathrm{Tr}(A)=a_{11}+a_{22},\quad \det A= a_{11}a_{22}-a_{12}a_{21}.$$
The combined matrix of the non-singular matrix $A$ is 
\begin{equation}\label{eq:cm2}
\mathcal{C}(A)=\frac{1}{\det A}\begin{pmatrix}a_{11}a_{22}&-a_{12}a_{21}\\-a_{12}a_{21}&a_{11}a_{22}\end{pmatrix}\end{equation}

\begin{proposition}\label{prop1}
Consider $A\in \mathrm{GL}_2(K)$, being $K$ a number field. The following statements hold
\begin{enumerate}
    \item $\mathrm{Tr}(\mathcal{C}(A))=\frac{2}{\mathrm{det}A}a_{11}a_{22}$
    \item $\mathrm{det}(\mathcal{C}(A))=\frac{1}{\mathrm{det}(A)}(a_{11}a_{22}+a_{12}a_{21})$
    \item The characteristic polynomial of $\mathcal{C}(A)$ is 
$$\lambda^2 - \left(\frac{2a_{11}a_{22}}{a_{11}a_{22} - a_{12}a_{21}}\right)\lambda + \frac{a_{11}a_{22}+a_{12}a_{21}}{a_{11}a_{22} - a_{12}a_{21}}$$
\item The eigenvalues of $\mathcal{C}(A)$ are given by $$\Lambda=\{1,\det\mathcal{C}(A)\}$$
\item The eigenvectors of $\mathcal{C}(A)$ are given by $$\left\{\begin{pmatrix}1\\1\end{pmatrix},\,\, \begin{pmatrix}1\\-1\end{pmatrix}\right\}$$
\item The Galois group $\mathrm{Gal}(L/K)$ of the characteristic polynomial is the identity group.
\item The combined matrix of $A$ is orthogonal diagonalizable: $$\mathcal{C}(A)=PDP^{T},\quad P=\frac{1}{\sqrt {2}}\begin{pmatrix}1 & 1\\ 1 & -1 \end{pmatrix},\quad D=\begin{pmatrix}1&0\\0&\det (\mathcal{C}(A))\end{pmatrix},\quad P^{T}=\frac{1}{\sqrt{2}}\begin{pmatrix}1&1\\1&-1\end{pmatrix}$$ Moreover, if $f$ is an analytic function in $1$ and in $\det (\mathcal{C}(A))$, then $$f(\mathcal{C}(A))=\frac{1}{2}\begin{pmatrix}1&1\\1&-1\end{pmatrix}\begin{pmatrix}f(1)&0\\0&f(\det (\mathcal{C}(A)))\end{pmatrix}\begin{pmatrix}1&1\\1&-1\end{pmatrix}$$
\end{enumerate}
\end{proposition}

\begin{proposition}\label{prop2}
Consider $A$ such that $A^2\in \mathrm{SL}_2(K)$, being $K$ a number field. The following statements hold
\begin{enumerate}
    \item $\mathrm{Tr}(\mathcal{C}(A))=\pm 2a_{12}a_{21}+2$
    \item $\mathrm{det}(\mathcal{C}(A))=\mathrm{Tr}(\mathcal{C}(A))-1$
    \item The characteristic polynomial of $\mathcal{C}(A)$ is 
$$p_{\mathcal{C}(A)}(\lambda)=\lambda^2 - \left(\pm 2a_{12}a_{21}+2\right)\lambda \pm 2a_{12}a_{21}+1$$
\item The eigenvalues of $\mathcal{C}(A)$ are given by $$\Lambda=\{1,\pm 2a_{12}a_{21}+1\}$$
\item The eigenvectors of $\mathcal{C}(A)$ are given by $$\left\{\begin{pmatrix}1\\1\end{pmatrix},\,\, \begin{pmatrix}1\\-1\end{pmatrix}\right\}$$
\item The Galois group of the characteristic polynomial  $\mathrm{Gal}(L/K)$ is the identity group.
\item The combined matrix of $A$ is orthogonal diagonalizable: $$\mathcal{C}(A)=PDP^{T},\quad P=\frac{1}{\sqrt{2}}\begin{pmatrix}1&1\\1&-1\end{pmatrix},\quad D=\begin{pmatrix}1&0\\0&\pm 2a_{12}a_{21}+1\end{pmatrix},\quad P^{T}=\frac{1}{\sqrt{2}}\begin{pmatrix}1&1\\1&-1\end{pmatrix}$$ Moreover, if $f$ is an analytic function in $1$ and in $\pm 2a_{12}a_{21}+1$, then $$f(\mathcal{C}(A))=\frac{1}{2}\begin{pmatrix}1&1\\1&-1\end{pmatrix}\begin{pmatrix}f(1)&0\\0&f(\pm 2a_{12}a_{21}+1\end{pmatrix}\begin{pmatrix}1&1\\1&-1\end{pmatrix}$$
\end{enumerate}

\end{proposition}

\subsection{Combined Matrices in $GL(3,K)$}

Consider the matrix $A\in \mathrm{GL}_3(K)$, being $K$ a number field, given by
$$
A=\begin{pmatrix}a_{11}&a_{12}&a_{13}\\
a_{21}&a_{22}&a_{23}\\a_{31}&a_{32}&a_{33}\end{pmatrix}$$
The determinant of $A$, denoted by $\det A$, is
$$\det A=a_{11}A_{11}+a_{12}A_{12}+a_{13}A_{13}.$$
The combined matrix of $A$ is
$$\mathcal{C}(A)=\frac{1}{\det A}\begin{pmatrix}a_{11}A_{11}& a_{12}A_{12}& a_{13}A_{13}\\ a_{21}A_{21}& a_{22}A_{22}&a_{23}A_{23}\\a_{31}A_{31}& a_{32}A_{32}& a_{33}A_{33}\end{pmatrix}$$

\begin{proposition}\label{prop3}
  Consider $A\in \mathrm{GL}_3(K)$, being $K$ a number field. The following statements hold
\begin{enumerate}
    \item $\mathrm{Tr}(\mathcal{C}(A))=1+\dfrac{2a_{11}a_{22}a_{33}-a_{12}a_{23}a_{31}-a_{13}a_{32}a_{21}}{\mathrm{det}(A)}$
    
    \item $\mathrm{Tr}(\mathcal{C}(A))=3-\dfrac{1}{\mathrm{det}(A)}\left[\displaystyle{\sum_{i=1}^{3}\sum_{j\neq i}^{3}a_{ij}A_{ij}}\right]$
    
     \item $\mathrm{det}(\mathcal{C}(A))= 1+\mathrm{Tr}(\mathcal{C}_{ij}(A))-\mathrm{Tr}(\mathcal{C}(A))$, where $\mathcal{C}_{ij}(A)$ is the cofactors matrix of $\mathcal{C}(A)$
   
    \item The characteristic polynomial of $\mathcal{C}(A)$ is 
$$\lambda^3-\mathrm{Tr}(\mathcal{C}(A))\lambda^2 + \left[\mathrm{Tr}(\mathcal{C}(A))+\mathrm{det}(\mathcal{C}(A))-1\right]\lambda + \mathrm{det}(\mathcal{C}(A))$$

\item The eigenvalues of $\mathcal{C}(A)$ are given by $$\Lambda=\{1,\lambda_{2}, \lambda_{3}\}$$
\\
where $\lambda_{2}, \lambda_{3}$ are solutions of $$\lambda^2 + \left[1-\mathrm{Tr}(\mathcal{C}(A))\right]\lambda + \mathrm{det}(\mathcal{C}(A))$$

\item A basis of eigenvectors of $\mathcal{C}(A)$ are given by $r_1$, $r_2$ and $r_3$, where $$r_{1}=\begin{pmatrix}a_{11}A_{11}-det(A)\\a_{11}A_{11}-det(A)\\a_{11}A_{11}-det(A)\end{pmatrix},r_{2}=\begin{pmatrix}a_{22}A_{22}-det(A)\lambda_2\\a_{22}A_{22}-det(A)\\a_{22}A_{22}-det(A)\lambda_2\end{pmatrix},r_{3}=\begin{pmatrix}a_{33}A_{33}-det(A)\lambda_3\\a_{33}A_{33}-det(A)\lambda_3\\a_{33}A_{33}-det(A)\end{pmatrix}$$

\end{enumerate}  
\end{proposition}

\begin{proposition}\label{prop4}
  Consider $A\in \mathrm{GL}_3(\mathbb R)$, a symmetric matrix. The following statements hold
\begin{enumerate}
 
\item $\mathrm{Tr}(\mathcal{C}(A))=3-\dfrac{2}{\mathrm{det}(A)}\left[\displaystyle{\sum_{i=1}^{2}\sum_{j=i+1}^{3}a_{ij}A_{ij}}\right]$

\item If the eigenvalues of $\mathcal{C}(A)$ given by $$\Lambda=\{1,\lambda_{2}, \lambda_{3}\}$$
\\
where $\lambda_{2}, \lambda_{3}$ are solutions of $$\lambda^2 + \left[1-\mathrm{Tr}(\mathcal{C}(A))\right]\lambda + \mathrm{det}(\mathcal{C}(A))$$ \\
are different, so the eigenvectors of $\mathcal{C}(A)$  given by $$\left\{r_{1},r_{2},r_{3}\right\}$$\\
they are orthogonal.

\item The combined matrix of $A$ is orthogonal diagonalizable: $$\mathcal{C}(A)=PDP^{T},\quad P=\begin{pmatrix}l_{1}&l_{2}&l_{3}\end{pmatrix}, l_{n}=\frac{r_{n}}{||r_{n}||}, n=1,2,3.$$\\
and $$D=\begin{pmatrix}1&0&0\\0&\lambda_{2} &0\\0&0&\lambda_{3}\end{pmatrix}$$

Moreover, if $f$ is an analytic function in $$\{1,\lambda_{2},\lambda_{3}\}$$ then $$f(\mathcal{C}(A))=P\begin{pmatrix}f(1)&0&0\\0&f(\lambda_{2})&0\\ 0&0&f( \lambda_{3})\end{pmatrix}P^{T}$$

\end{enumerate}  
\end{proposition}

\begin{corollary}
 Consider $A\in \mathrm{GL}_3(\mathbb Q)$, a symmetric matrix. If \ $\mathrm{Tr}(\mathcal{C}_{ij}(A))=-1$, then the following statements hold.  
 
\begin{enumerate} 
\item $\mathrm{Tr}(\mathcal{C}(A))=-\mathrm{det}(\mathcal{C}(A))$

\item The eigenvalues of $\mathcal{C}(A)$ given by $$\Lambda=\{1,-1,-\mathrm{det}(\mathcal{C}(A))\}$$
\item The Galois group of the characteristic polynomial  $\mathrm{Gal}(L/\mathbb Q)$ is the identity group.

\end{enumerate}  
\end{corollary}

\subsection{Combined Matrices in $GL(n,K)$}

Consider the matrix $A\in \mathrm{GL}_n(K)$, being $K$ a field number, given by
$$
A=\begin{pmatrix}a_{11}&a_{12}&...&a_{1n}\\
a_{21}&a_{22}&...&a_{2n}\\\vdots&\vdots&\vdots&\vdots&\\a_{n1}&a_{n2}&...&a_{nn}\end{pmatrix}$$
The combined matrix of $A$, is 
$$\mathcal{C}(A)=\frac{1}{\det A}\begin{pmatrix}a_{11}A_{11}& a_{12}A_{12}&...&a_{1n}A_{1n} \\ a_{21}A_{21}& a_{22}A_{22}&...&a_{2n}A_{2n}\\\vdots&\vdots&\vdots&\vdots&\\a_{n1}A_{n1}&a_{n2}A_{n2}&...&a_{nn}A_{nn}\end{pmatrix}$$

\begin{proposition}\label{randy}
  Consider $A\in \mathrm{GL}_n(K)$, being $K$ a number field. The following statements hold
\begin{enumerate}
    
    \item $\mathrm{Tr}(\mathcal{C}(A))=n-\dfrac{1}{\mathrm{det}(A)}\left[\displaystyle{\sum_{i=1}^{n}\sum_{j\neq i}^{n}a_{ij}A_{ij}}\right]$
    
     \item $\mathrm{det}(\mathcal{C}(A))= \displaystyle{\sum_{k=1}^{n}(-1)^{k+1}Tr(C^{(k)})}$,

where $C^{(k)}$ is the cofactor matrix of $k$ order of  $\mathcal{C}(A)$
   
    \item The characteristic polynomial of $\mathcal{C}(A)$ is 
$$ P_{n}(\lambda)=\sum_{k=0}^{n}(-1)^{k}\lambda^{k}Tr(C^{(k)})$$
 where $C^{(k)}$ is the cofactor matrix of k order of  $\mathcal{C}(A)$, $Tr(C^{(0)})=det(\mathcal{C}(A))$ and $Tr(C^{(n)})=1$

\item The eigenvalues of $\mathcal{C}(A)$ are given by $$\Lambda=\{1,\lambda_{2},..., \lambda_{n}\}$$
where $\lambda_{2},..., \lambda_{n}$ are solutions of the polinomial $$ Q_{n-1}(\lambda)=\sum_{r=1}^{n}\sum_{k=0}^{r-1}(-1)^{k}Tr(C^{(n-k)})\lambda^{n-r}$$

\item The eigenvectors of $\mathcal{C}(A)$ are given by $$\left\{r_{1}=\begin{pmatrix}a_{11}A_{11}-det(A)\\a_{11}A_{11}-det(A)\\\vdots\\a_{11}A_{11}-det(A)\end{pmatrix},r_{2}=\begin{pmatrix}a_{22}A_{22}-det(A)\lambda_2\\a_{22}A_{22}-det(A)\\a_{22}A_{22}-det(A)\lambda_2\\\vdots\end{pmatrix},...,r_{n}=\begin{pmatrix}a_{nn}A_{nn}-det(A)\lambda_n\\a_{nn}A_{nn}-det(A)\lambda_n\\\vdots\\a_{nn}A_{nn}-det(A)\end{pmatrix}\right\}$$

\end{enumerate}  
\end{proposition}

\begin{proposition}\label{randy1}
  Consider $A\in \mathrm{GL}_n(\mathbb R)$, a symmetric matrix. The following statements hold
\begin{enumerate}
    
\item $\mathrm{Tr}(\mathcal{C}(A))=n-\dfrac{2}{\mathrm{det}(A)}\left[\displaystyle{\sum_{i=1}^{n-1}\sum_{j=i+1}^{n}a_{ij}A_{ij}}\right]$
   
\item If the eigenvalues of $\mathcal{C}(A)$ are given by $$\Lambda=\{1,\lambda_{2},..., \lambda_{n}\}$$
\\
where $\lambda_{2},..., \lambda_{n}$ are solutions of the polinomial $$ Q_{n-1}(\lambda)=\sum_{r=1}^{n}\sum_{k=0}^{r-1}(-1)^{k}Tr(C^{(n-k)})\lambda^{n-r}$$

are different, so the eigenvectors of $\mathcal{C}(A)$ given by $$\left\{r_{1}, r_{2},..., r_{n}\right\}$$
they are orthogonal.

\item The combined matrix of $A$ is orthogonal diagonalizable: $$\mathcal{C}(A)=PDP^{T},\quad P=\begin{pmatrix} l_{1}&l_{2}&...&l_{n}\end{pmatrix}, l_{n}=\frac{r_{n}}{||r_{n}||}, n=1,2,...$$\\
and $$D=\begin{pmatrix}1&0&...&0\\0&\lambda_{2} &...&0\\\vdots&\vdots&\vdots&\vdots&\\0&0&...&\lambda_{n}\end{pmatrix}$$ 
Moreover, if $f$ is an analytic function in $$\{1,\lambda_{2},...,\lambda_{n}\}$$ then $$f(\mathcal{C}(A))=P\begin{pmatrix}f(1)&0&...&0\\0&f(\lambda_{2}) &...&0\\\vdots&\vdots&\vdots&\vdots&\\0&0&...&f(\lambda_{n})\end{pmatrix}P^{T}$$

\end{enumerate}  
\end{proposition}

\subsection{Combined matrices over the General Linear Group}
Here we consider combined matrices over the General Linear Group. We start with the following Lemma that involves Hadamard product and Reversing of matrices.
\begin{lemma}\label{lp}
Let $A,B\in\mathrm{GL}_n(K)$, $\circ$ be the Hadamard product and let $\mathcal{R}$ be Reversing over $\mathrm{GL}_n(K)$:
$$\mathcal{R}:\,\, (\mathrm{GL}_n(K),\circ) \longrightarrow (\mathrm{GL}_n(K),\circ).$$
Then, Reversing is a group automorphism from $(\mathrm{GL}_n(K),\circ)$ to $(\mathrm{GL}_n(K),\circ)$.
\end{lemma}
\begin{proof}
We start proving that $\mathcal{R}$ is homomorphism. Thus, $\mathcal{R}(A\circ B)=M_{\mathcal{R}_c}(A\circ B)M_{\mathcal{R}_r}$. By definition of Reversing it is satisfied $I_n=M_{\mathcal{R}_c}\cdot M_{\mathcal{R}_r}=M_{\mathcal{R}_r}\cdot M_{\mathcal{R}_c}$ we have $\mathcal{R}(A\circ B)=(M_{\mathcal{R}_c}\cdot M_{\mathcal{R}_r}\cdot M_{\mathcal{R}_c}) (A\circ B) M_{\mathcal{R}_r}=(M_{\mathcal{R}_c}A\cdot M_{\mathcal{R}_r})\circ (M_{\mathcal{R}_c} B\cdot M_{\mathcal{R}_r})$ and we conclude $\mathcal{R}(A\circ B)=\mathcal{R}(A)\circ \mathcal{R}(B)$. Now, we prove that $\mathcal{R}$ is Monomorphism. Thus, assuming  $\mathcal{R}(A)=\mathcal{R}(B)$ and applying again $\mathcal{R}$, we obtain $\mathcal{R}^2(A)=\mathcal{R}^2(B)$ and therefore $A=B$. We see that $\mathcal{R}$ is Epimorphism because $\mathcal{R}(A)\in \mathrm{GL}_n(K)$ for all $A\in \mathrm{GL}_n(K)$. In this way we proved that $\mathcal{R}$ is Isomorphism and because it is Endomorphism by definition, we conclude that $\mathcal{R}$ is a group Automorphism from $(\mathrm{GL}_n(K),\circ)$ to $(\mathrm{GL}_n(K),\circ)$.
\end{proof}
\begin{theorem}
Let $\mathcal{C}(A)$ be the combined matrix of $A\in\mathrm{GL}_n(K)$ as in Eq. \eqref{eq:dcm}. Let $L=K[\lambda_1,\ldots,\lambda_n]$ be the extension of $K$. The following statements hold
$$
\begin{array}{ll}
1.&\mathcal{C}(A)=[m_{ij}]_{n\times n},\quad m_{ij}=\displaystyle{\frac{(-1)^{i+j}}{\mathrm{det}A}a_{ij}A_{ij}} (arriba)\\
2.&\mathrm{Tr}(\mathcal{C}(A))=\frac{1}{\det A}\displaystyle{\sum_{k=1}^na_{ii}A_{ii}}\\
3.& (1,v_1)\in \Lambda(\mathcal{C}(A))\times \mathcal{V}(\mathcal{C}(A)),\,\, v_1=\displaystyle{\sum_{k=1}^ne_k}\\
4. & \mathrm{Gal}(L/K)\leq G,\quad G\cong S_{n-1}\\
5. & \mathcal{C}(\mathcal{R}(A))=\mathcal{R}(\mathcal{C}(A))
\end{array}
$$
\end{theorem}
\begin{proof}
We proceed according to each item.
\begin{enumerate}
    \item Due to $A\in \mathrm{GL}_n(K)$, $A$ is non-singular and formula \eqref{eq:dcm} is valid. Now, by direct application of basic properties of transpose matrix and inverse matrix we obtain the result.
    \item It follows by direct application of item 1.
    \item  By induction and basic properties of determinants we have $P_{\mathcal{C}(A)}(1)=0$ and therefore $\lambda_1=1$ is an eigenvalue of $\mathcal{C}(A)$. Due to vector $v_1$ corresponds to $(1,1,\ldots,1)^T\in K^n$, we obtain $\mathcal{C}(A)\cdot v_1=v_1$ and therefore $v_1$ is an eigenvector of $\mathcal{C}(A)$.
    \item By item 3 we have $p_{\mathcal{C}(A)}(1)=0$, thus $p_{\mathcal{C}(A)}(\lambda)=(\lambda-1)q_{n-1}(\lambda)$, where $q_{n-1}$ is a polynomial of degree $n-1$. Therefore the extension $L$ is obtained as $K[\lambda_2,\ldots,\lambda_{n}]$, which implies that $\mathrm{Gal}(L/K)$ is a subgroup of $S_{n-1}$.
    \item It follows directly from Lemma \ref{lp}. Thus, Reversing of the combined matrix of $A$ is the combined matrix of Reversing of $A$.
\end{enumerate} 
\end{proof}
\subsection{Combined matrices over proper subgroups of General Linear Group}
Here we consider combined matrices over some proper subgroups of $\mathrm{GL}_n(K)$ such as $\mathrm{SL}_n(K)$, triangular matrices, orthogonal matrices, among others. We start with the following Lemma.
\begin{lemma}\label{lp2}
Let $(T_n^+,\cdot)$ and $(T_n^-,\cdot)$ be the groups of $n\times n$ upper triangular and lower triangular matrices respectively with the usual product of matrices. Then $T_n^+\circ T_n^-=\mathrm{diag}(T_n^+)\cdot \mathrm{diag}(T_n^-)$, where $\circ$ is the Hadamard product of matrices and $\mathrm{diag}$ denotes the diagonal matrix.
\end{lemma}
\begin{proof}
Consider $T_n^+=[a_{ij}]_{n\times n}$ and $T_n^-=[b_{ij}]_{n\times n}$ such that $a_{ij}=0$ for all $i<j$ and $b_{ij}=0$ for all $i>j$. Thus, $T_n^+\circ T_n^-=[a_{ij}\cdot b_{ij}]_{n\times n}=[a_{ii}\cdot b_{ii}]_{n\times n}=[a_{ii}]_{n\times n}\circ [b_{ii}]_{n\times n}$ which lead us to the result. 
\end{proof}
\begin{proposition}
 The following statements hold.
\begin{enumerate} 
\item $\mathcal{C}(T_n)=I_n$
\item $\Lambda(\mathcal{C}(T_n))=\{1\}$.
\item $\mathcal{C}(AB)=\mathcal{C}(A)\mathcal{C}(B)$, being $A,B$ triangular matrices.
\end{enumerate}
\end{proposition}
\begin{proof}
We proceed according to each item.
\begin{enumerate}
    \item By definition of combined matrix of $A$ we have $\mathcal{C}(A)=A\circ A^{-T}$. Now, inverse of upper (resp. lower) triangular matrix is upper (resp. lower) triangular matrix. In the same way, transpose of upper (resp. lower) triangular matrix is a lower (resp. upper) triangular matrix and the determinant of any triangular (lower and upper) is the product of the diagonal entries. Thus, by Lemma \ref{lp2} we have that $\mathcal{C}(T_n^{\pm})$ are diagonal matrices and $T_n^{\pm} \circ (T_n^{\pm})^{-T}$ and because we divide by the product of the diagonal entries, we obtain $\mathcal{C}(T_n^{\pm})=I_n$.
    \item It follows directly from item 1 because the identity matrix has 1 as eigenvalue of algebraic multiplicity $n$.
    \item It follows due to the $\ker (\mathcal{C})$ for triangular matrices is the identity.
\end{enumerate}
\end{proof}
The following result involves the \emph{Orthogonal Group} over $K$, that is, the group of matrices with entries in $K$ such that the product of a matrix with its transpose is the identity. We observe that Orthogonal Group is a subgroup of $\mathrm{SL}_n(K)$ 
\begin{proposition}
The matrix $A$ is an orthogonal matrix if and only if $\mathcal{C}(A)=A\circ A$. %Moreover, $\mathcal{C}(A)$ is symmetric.
\end{proposition}
\begin{proof}
Due to $A^{-1}=A^T$ we have that $A^{-T}=A$ and therefore $\mathcal{C}(A)=A\circ A$.
\end{proof}
\noindent \begin{example}
Consider the following matrix $$A=-\frac{1}{3}\begin{pmatrix}2& -2& 1\\1& 2& 2\\ 2& 1& -2\end{pmatrix},\quad A^{-1}=A^T=-\frac{1}{3}\begin{pmatrix}2& 1& 2\\-2& 2& 1\\ 1& 2& -2\end{pmatrix},$$
the combined matriz of $A$ is $$\mathcal{C}(A)=\frac{1}{9}\begin{pmatrix}4& 4& 1\\1& 4& 4\\ 4& 1& 4\end{pmatrix}.$$
\end{example}
%The following result is a consequence of previous proposition.
%\begin{proposition}
%Let $A=[a_{ij}]_{n\times n}$ be a symmetric non-singular matrix with positive real entries. There exists a matrix $B=[b_{ij}]_{n\times n}$ belonging to the Orthogonal Group satisfying 
%$$A=\mathcal{C}(B),\quad b_{ij}=\pm \sqrt{\frac{a_{ij}}{\det A}}$$
%\end{proposition}
%\begin{remark}
%Concerning the previous proposition, we recall that orthogonal matrices are not necessarily symmetric. For example, considering $$A=\begin{pmatrix}1& 0& 1\\3& -1& 2\\ 5& -1& 3\end{pmatrix},\quad \det A= 1,\quad A^{-1}=\begin{pmatrix}-1& -1& 1\\1& -2& 1\\ 2& 1& -1\end{pmatrix}$$ 
%,\end{remark}
%\subsection{Combined matrices over proper subgroups of $\mathrm{GL}_n(K)$}

%\subsection{Results of $\mathcal{C}(A)$ for particular cases}
%In this section we consider algebraic aspects of characteristic polynomials of $n\times n$ combined matrices given in Eq. \eqref{eq:dcm} being $2\leq n\leq $. 

\section{Final remarks}

In this paper, we analyzed various algebraic aspects of combined matrices. We began by relating the Hadamard product with the reversing of matrices in the General Linear Group. Subsequently, we presented results concerning the spectrum of combined matrices, explicitly providing some elements of the spectrum and the set of eigenvectors. Specifically, we found that 1 is an eigenvalue of any combined matrix, and its corresponding eigenvector is the sum of the elements of the canonical basis. Some results were provided for any \(n \times n\) matrix, and we examined the case of \(2 \times 2\) matrices and specific cases for \(3 \times 3\) matrices. The Galois group is the identity for the characteristic polynomial of any \(2 \times 2\) combined matrix.

For \(n > 3\), the eigenvectors will depend on the entries of the original matrix \(A\), except for the eigenvector corresponding to the eigenvalue 1. Further analysis for \(n > 3\) is in progress.

Some aspects to be considered as extensions of these results are as follows:

\begin{itemize}
    \item Obtaining algebraic characterization for \(n \times n\) combined matrices, \(n > 3\).
    \item Relations between quantum mechanics and combined matrices (see \cite{amw11}).
    \item Relations between Pasting and Reversing with combined matrices (see \cite{acr13,acr10,acma16}.
    \item Relations between differential Galois theory and combined matrices (see \cite{a10,amw11}).
    \item Relation of tensor analysis and combined matrices.
\end{itemize}

\bibliographystyle{alpha}  
%\bibliography{ams9final}  %%% Remove comment to use the external .bib file (using bibtex).


\begin{thebibliography}{1}
  
  \bibitem{a10} Acosta-Humánez, P. B. (2010). \emph{Galoisian Approach to Supersymmetric Quantum Mechanics: The integrability analysis of the Schrödinger equation by means of differential Galois theory}. VDM Publishing.

  \bibitem{ac16} Acosta-Humánez, P \& Chuquen, A. (2016). \emph{Pasting and Reversing approach to matrix theory}. Journal of Algebra, Number Theory and Applications, \textbf{38}(6) 535--559

  \bibitem{acr10} Acosta-Humanez, P., Chuquen, A., \& Rodr\'iguez, A. (2013). \emph{Pasting and Reversing Operations over some Rings}. Bolet\'in de Matem\'aticas, \textbf{17}(2), 143--164.
  
  \bibitem{acr13} Acosta-Humanez, P., Chuquen, A., \& Rodr\'iguez, A. (2013). \emph{Pasting and Reversing Operations over some Vector Spaces}. Bolet\'in de Matem\'aticas, \textbf{20}(2), 145--161.

\bibitem{acma16} Acosta-Hum{\'a}nez, P \& Mart{\'\i}nez-Castiblanco, O. (2016
  \emph{Simple Permutations with Order $4 n+ 2$ by Means of Pasting and Reversing},Qualitative theory of dynamical systems, \textbf{15}(1), 181--210.
  
  \bibitem{amw11} Acosta-Humánez, P. B., Morales-Ruiz, J. J., \& Weil, J. A. (2011). \emph{Galoisian approach to integrability of Schrödinger equation}. Reports on Mathematical Physics, \textbf{67}(3), 305--374.
  
  \bibitem{aps19} Alonso, P., Peña, J. M., \& Serrano, M. L. (2019). \emph{Combined matrices of almost strictly sign regular matrices}. Journal of computational and applied mathematics, \textbf{354}, 144--151.

\bibitem{bggss23}Bru, R., Gassó, M. T., Giménez, I., Santana, M., \& Scott, J. (2023). \emph{Combined matrix of diagonally equipotent matrices}. Special Matrices, \textbf{11}(1), 20230101.

\bibitem{bggs16} Bru, R., Gassó, M. T., Giménez, I., \& Santana, M. (2016). \emph{Combined matrices of sign regular matrices}. Linear Algebra and Its Applications, 4\textbf{98}, 88--98.

\bibitem{bggs14} Bru, R., Gassó, M. T., Giménez, I., \& Santana, M. (2014). \emph{Nonnegative combined matrices}. Journal of Applied Mathematics, \textbf{2014}(1), 182354.

\bibitem{bgs22} Bru, R., Gassó, M. T., \& Santana, M. (2022). \emph{Combined matrices and conditioning}. Applied Mathematics and Computation, \textbf{412}, 126549.

\bibitem{caa13} Charris Castañeda, J., Aldana Gómez, B. y Acosta-Humánez, P. (2013). \emph{Algebra. Fundamentos, Grupos, Anillos, Cuerpos y Teoría de Galois}. Colección Julio Carrizosa Valenzuela \textbf{16}, Academia Colombiana de Ciencias Exactas, Físicas y Naturales.

\bibitem{f17} Fiedler, M. (2017). \emph{Some results on combined matrices}. Banach Center Publications, \textbf{112}, 99--105.

\bibitem{fm11} Fiedler, M., \& Markham, T. L. (2011). \emph{Combined matrices in special classes of matrices}. Linear algebra and its applications, \textbf{435}(8), 1945-1955.

\bibitem{gggss21} Gassó, M. T., Gil, I., Giménez, I., Santana, M., \& Segura, E. (2021). \emph{Diagonal entries of the combined matrix of sign regular matrices of order three}. Proyecciones (Antofagasta), \textbf{40}(1), 255--271.

\end{thebibliography}

\end{document}